\documentclass{amsart}
\usepackage{amsthm,amstext,amsmath,amscd,amssymb,latexsym}
\usepackage{color}
\usepackage[matrix,arrow]{xy}
\usepackage{amsfonts,enumerate,array}
\usepackage[colorlinks=true]{hyperref}
\usepackage{todonotes}
	\usepackage{verbatim}
\usepackage[toc,page]{appendix}

\newcommand{\PP}{{\mathbb P}}

\newcommand{\rr}{{\mathbb R}}
\newcommand{\QQ}{{\mathbb Q}}

\newcommand{\LL}{{\mathcal L}}

\newcommand{\Sing}{{\rm{Sing}}}

\newcommand{\ls}{{\mathcal{L}}}

\DeclareMathOperator{\edim}{edim}
\DeclareMathOperator{\ldim}{ldim}

\DeclareMathOperator{\vdim}{vdim}

\DeclareMathOperator{\mult}{mult}
\DeclareMathOperator{\Pic}{Pic}
\DeclareMathOperator{\Cr}{Cr}
\DeclareMathOperator{\Eff}{Eff}

\DeclareMathOperator{\Mov}{Mov}
\DeclareMathOperator{\NS}{N}
\DeclareMathOperator{\J}{J}
\DeclareMathOperator{\Cox}{Cox}
\DeclareMathOperator{\sldim}{\sigma ldim}

\newcommand{\paper}{: \begin{it}}
\newcommand{\jour }{, \end{it}}

\newtheorem{theorem}{Theorem}[section]
\newtheorem{lemma}[theorem]{Lemma}
\newtheorem{proposition}[theorem]{Proposition}
\newtheorem{corollary}[theorem]{Corollary}
\newtheorem{conjecture}[theorem]{Conjecture}

\theoremstyle{definition}
\newtheorem{definition}[theorem]{Definition}

\newtheorem{example}[theorem]{Example}

\theoremstyle{remark}
\newtheorem{remark}[theorem]{Remark}

\numberwithin{equation}{section}

\title{On the effective cone of $\PP^n$ blown-up at $n+3$ points}

\author{Maria Chiara Brambilla}
\email{{\tt brambilla@dipmat.univpm.it}}
\address{Universit\`a Politecnica delle Marche, 
via Brecce Bianche, I-60131 Ancona, Italy}

\author{Olivia Dumitrescu}
\email{{\tt  dumitrescu@math.uni-hannover.de}}
\address{Institut f\"ur Algebraische Geometrie GRK 1463, Welfengarten 1, 30167 Hannover, Germany}

\author{Elisa Postinghel}
\email{{\tt elisa.postinghel@wis.kuleuven.be}}
\address{KU Leuven, Department of Mathematics, Celestijnenlaan 200B, 3001 Heverlee,
Belgium}

\thanks{The first author is partially supported by MIUR and INDAM.
The second author is a member of 
the Simion Stoilow Institute of Mathematics of the 
Romanian Academy.
The third author is supported by the Research Foundation - Flanders (FWO)}


\begin{document}

\begin{abstract}
 We compute the facets of the effective and movable
cones of divisors on the blow-up of $\PP^n$ at $n+3$ points in general position.
Given any linear system of hypersurfaces of $\PP^n$ based at $n+3$ multiple points in general position,  
we prove that the secant varieties to the rational normal 
curve of degree $n$ passing through the points, as well as their joins with linear subspaces spanned by some of the points,
 are cycles of the base locus and we compute their multiplicity. 
{We conjecture that a linear system with $n+3$ points is linearly special only if it contains such subvarieties in the base locus
and we give a new formula for the expected dimension.}
\end{abstract}

\maketitle

\section{Introduction}

Let $\LL=\LL_{n,d}(m_1,\ldots,m_s)$ denote the linear system of 
hypersurfaces of degree $d$ in $\PP^n$ passing through a collection of $s$ points in general position with multiplicities
at least
$m_1,\ldots,m_s$. A classical question is to compute the dimension of $\ls$. 
A parameter count provides a lower bound: the {\em (affine) virtual dimension} of $\LL$ is denoted by
\begin{equation}\label{virtual dim}
\vdim(\LL)=\binom{n+d}{n}-\sum_{i=1}^s\binom{n+m_i-1}{n}
\end{equation}
and the {\em (affine) expected dimension} of $\LL$ is $\edim(\LL)=\max(\vdim(\LL),0)$. 
If the dimension of $\LL$ is strictly  greater that the expected dimension 
we say that $\LL$ is \emph{special}.

{
The \emph{dimensionality problem}, that is the classification of all special linear systems, is still open in general, in spite of intensive investigation by many authors. 
In particular, the study of linear systems requires information on the \emph{effective cone} $\Eff_{\rr}(X)$ of divisors
on the blow-up $X$ of $\PP^n$ at the given points. 
Also the computation of the effective cone is in general a difficult task.}

{Let us overview now some known results.}
In the planar case, the Segre-Harbourne-Gimigliano-Hirschowitz Conjecture
describes all effective special linear systems. 
{It predicts that a linear system in $\PP^2$ is special only if it contains in its base locus $(-1)$-curves.}
 In particular, it conjectures the effective cone of divisors on $\PP^2$ blown-up at points (see \cite{Ciliberto, Gimigliano, Harbourne3, Hir2}). 
On the negative side, we mention Nagata's Conjecture that predicts the nef 
cone of linear systems in the blown-up plane at general points. 
Even for the case of dimension two in spite of many partial results, both conjectures are open in general (see \cite{CHMR}).

In the case of $\PP^3$, Laface and Ugaglia Conjecture states that a Cremona reduced special linear system in $\PP^3$
contains in its base locus either lines 
or the unique quadric surface determined by nine general points (see e.g.\ \cite{laface-ugaglia-TAMS, laface-ugaglia-elem}).
The base locus lemma for the quadric in $\PP^3$ is difficult; it is related to Nagata's Conjecture for ten points in 
$\PP^2$ (see \cite{BraDumPos2}). The degeneration technique introduced by Ciliberto and Miranda
 (see e.g. \cite{BraDumPos2, CM1}) is a successful method in the study of interpolation problems in higher dimensions.

In the case of $\PP^n$ general results are rare and few things are known. 
The well-known Alexander-Hirschowitz Theorem \cite{AlHi} classifies completely the case of double points 
(see \cite{ale-hirsch, Ch1, Ch2, Po} for more recent and simplified proofs). 
In general, besides the computation of some sporadic
 examples
and the formulation of conjectures about the speciality (see e.g.\ \cite{bocci1,bocci2}),
very little {is known.}

{
The important feature of special linear systems interpolating double points in $\PP^n$ is that every element is singular along a positive dimensional subvariety containing the points. 
As conjectured by Ciliberto and Miranda (see \cite[Conjecture 6.4]{Ciliberto}), it is expected that this is the case also for higher multiplicities.
So it is natural to look at the {\it base locus} of special linear systems and try to understand the possible connection with the speciality. We  call {\em obstructions} the subvarieties that whenever contained with multiplicity in the {base locus} of a linear system force it to be special.}

{
In \cite{BraDumPos1, DumPos} the authors started a systematic study of special linear systems from this point of view and considered in particular the case of linear obstructions.
Taking into account the contribution of all the linear cycles of the base locus, a new notion of expected dimension can be given, the {\it linear expected dimension} $\ldim(\LL)$ (see Definition \ref{new-definition}).
We say that a system $\LL$ is \emph{linearly special} (resp.\ \emph{linearly non-special}) if its dimension differs from (resp.\ equals) the linear expected dimension. 
In other words a linear system is linearly special if its speciality cannot be explained completely by linear obstructions.  
Any linear system with $s\le n+2$ base points is linearly non-special.
This was proved in \cite{BraDumPos1, DumPos} by means of a complete cohomological classification of strict transforms in subsequently blown-up spaces of  linear systems, see Section \ref{resultsBraDumPos} for an account. 
}

The first instance of a non-linear cycle of the base locus of a special linear system is the rational normal curve of degree $n$ through $n+3$ general points of $\PP^n$. 
The well-known Veronese Theorem (often referred to as the Castelnuovo Theorem)
tells us that there exists exactly one such a curve.
In  $\mathbb{P}^2$ an instance of this is the unique conic through five points. 
In this article we focus on special linear system obstructed by rational normal curves and related varieties.
We prove first a \emph{base locus lemma} (Lemma \ref{base locus lemma rnc})
for linear systems with arbitrary number of general points
{ which describes the following 
non-linear cycles of the base locus:
the rational normal curve, its secant varieties and the cones over them.} 
For instance, the fixed cubic {hypersurface} of $\PP^4$ interpolating $7$ double points, that appears as one of the exceptions in the 
Alexander-Hirschowitz theorem, is the variety of secant lines to the rational normal curves given by the seven points.

{We expect that }when the multiplicity of containment in the base locus is high enough with respect to the degree, 
those cycles forces the linear system to be special.
{More precisely, we give} a conjectural formula (Definition
 \ref{new definition rnc}, Conjecture \ref{conjecture}) 
for the dimension of linear systems based at $n+3$ points. 
The formula in Definition \ref{new definition rnc} takes into account 
 the contribution of the linear cycles and also that of the normal curves and related cycles {in the base locus}.
In Section \ref{evidences} we prove that this conjecture holds for $n=2,3$ and for general $n$ in a number of interesting families 
of homogeneous linear systems.

{As a consequence of our analysis,} we deduce {the main result of this paper, that is} an explicit description of all effective 
divisors in $X$,  the  blown-up $\PP^n$ at $n+3$ points. 
{We give in Theorem \ref{effectivity lemma Pn n+3} a list of inequalities that define the effective cone of $X$,  
and we also describe the {movable cone} of $X$, 
Theorem \ref{thm movable cone}.}

We mention that a new approach to the dimensionality problem for $s=n+3$ points
was introduced in  \cite{sturmfels} and their analysis relies on sagbi bases.
For $s\le n+3$, in  \cite{CT, Mukai2} it was proved that the blow-up 
of $\PP^n$ at $s$ points in general position 
 is a \emph{Mori dream space}. In particular Castravet and Tevelev \cite{CT}
gave the rays of the effective cone, see Section \ref{movable cone} for more details. 
What is interesting is the fact that Castravet and Tevelev's extremal rays can be formulated 
in terms of hypersurfaces that are either secant varieties to the rational normal curve through the $n+3$ points or 
 their joins with linear subspaces spanned by the points,
see Section \ref{secants}.
In this paper we show that in fact both the effective cone 
(Theorem \ref{effectivity lemma Pn n+3})
and movable cone (Theorem \ref{thm movable cone}) of the blown-up space, and the dimensionality 
problem, depend exclusively on these secant varieties seen as cycles of arbitrary codimension in $\PP^n$.
{
Moreover there is a bijection between the $2^{n+2}$  weights of the
half-spin representations of $\mathfrak{so}_{2(n+3)}$ and the generators of the Cox
ring of the blow-up of $\mathbb{P}^n$ at $n+3$ points in general position
(see \cite{Dolgachev, CT, sturmfels-velasco}). In \cite{sturmfels-velasco} in particular this bijection interprets the latter
space as a spinor variety. It would be interesting to establish a
dictionary that translates the language of secant varieties into that of
spinor varieties.
}

The article is organized as follows.
In Section \ref{resultsBraDumPos} we give an account on the notion of linear speciality
 \cite{BraDumPos1,DumPos}.

In Section \ref{secants} we  give a geometric description of the rational normal curves and (cones over) their secants and we give an 
interpretation in terms of divisors of those among them that are of codimension $1$, by means of Cremona transformations
of $\PP^n$. In particular, these divisors are the Castravet-Tevelev rays generating the effective cone.

In Section \ref{base locus lemma for secants} we prove the base locus lemma for rational normal curves and related cycles, Lemma \ref{base locus lemma rnc}.

In Section \ref{effectivity for secants} we describe the effective and movable cones of $X$, Theorem \ref{effectivity lemma Pn n+3}
and Theorem \ref{movable cone}.

In Section \ref{new exp dim with secants} we introduce the new notion of expected dimension, $\sldim$ (Definition \ref{new definition rnc}),
and  
state our Conjecture \ref{conjecture}, 
exhibiting a list of evidences 
in Section \ref{evidences}.

\subsection*{Acknowledgements} 
The authors would like to thank the Research Center FBK-CIRM  Trento
for the hospitality and financial support during the stay for the Summer School ``An interdisciplinary approach to tensor decomposition'' (Summer 2014)
and during their one month ``Research in Pairs'' program  (Winter 2015).
We are grateful to Ana-Maria Castravet for pointing out general aspects of this work and further research directions and to 
 Cinzia Casagrande for pointing out some inaccuracies in a previous version of this paper.
{We thank the referee for his/her helpful comments.}

\section{
{Linear speciality of linear systems}}
\label{resultsBraDumPos}

{Given a linear system $\LL$, we say that {\it its base locus  contains a subvariety $L$ with multiplicity $k$} if any  hypersurface in $\LL$ has multiplicity at least $k$ along each point of $L$.}

Let $\ls=\ls_{n,d}(m_1,\dots,m_s)$
be a non-empty linear system,  let 
$I(r)\subseteq\{1,\dots,s\}$ be any  multi-index
   of length $|I(r)|=r+1$, for $0\le r\le \min(n,s)-1$ and denote
by $L_{I(r)}$ the unique $r$-linear cycle through the points $p_i$, for $i \in I(r)$.
Set 
\begin{equation}\label{mult k} 
 k_{I(r)}=\max \left(\sum_{i\in I(r)}m_i -rd,0\right).
\end{equation}

It is an easy consequence of B\'ezout's Theorem that if $k_{I(r)}>0$ 
then all elements of $\ls$ vanish along $L_{I(r)}$. 

In \cite{BraDumPos1} a (sharp)  base locus lemma for linear cycles, that we will refer to as \emph{linear base locus lemma},
 for linear systems with at most $n+2$ points was proved
 and later generalized in \cite{DumPos} to linear systems with arbitrary numbers of points.
We summarize the content of the two above mentioned results in the following

\begin{lemma}[{\cite[Proposition 4.2]{DumPos}}]\label{linear base locus}
For any non-empty linear system $\ls$ with arbitrary number of points and for any $0\leq r\leq n-1$,
 the multiplicity of containment of the cycle $L_{I(r)}$ in the base locus of $\ls$ is $k_{I(r)}$.
\end{lemma}

When the order of vanishing is high, precisely when $k_{I(r)}>r$, then $L_{I(r)}$ provides obstruction to the non-speciality.
This observation yields the following definition of expected dimension.

\begin{definition}[{\cite[Definition 3.2]{BraDumPos1}}]\label{new-definition}
The {\em (affine) linear virtual dimension} of $\LL$ is the number
\begin{equation}\label{linvirtdim}
\sum_{r=-1}^{s-1}\sum_{I(r)\subseteq \{1,\ldots,s\}} (-1)^{r+1}\binom{n+k_{I(r)}-r-1}{n}.
\end{equation}
 where we set $I(-1)=\emptyset$ and $k_{I(-1)}=d.$

The {\em (affine) linear expected dimension} of $\LL$, {denoted by $\ldim(\LL)$}, is $0$ if $\LL$ is contained in a linear system whose 
linear virtual dimension is negative, otherwise is the maximum between the linear virtual dimension of $\LL$ and $0$.
\end{definition}

In (\ref{linvirtdim}), the number $(-1)^{r+1}{{n+k_{I(r)}-r-1}\choose{n}}$ computes the 
contribution of the linear 
cycle $L_{I(r)}\cong\mathbb{P}^r$ spanned by the points $p_{i_j}$, $i_j\in I(r)$. 
If all the numbers $k_{I(r)}$ are zero, the linear virtual dimension (\ref{linvirtdim}) equals the virtual dimension \eqref{virtual dim} of $\ls$.
Asking whether the dimension of a
 given linear system equals its linear expected dimension is a refinement of the classical question of asking whether the dimension 
equals the expected dimension. 
{We say that a linear system $\LL$ is {\it linearly special} if $\dim(\LL)\neq\ldim(\LL)$. On the other hand}
a linear system is called {\em linearly non-special} (or {\em only linearly obstructed})
if its dimension equals the linear expected dimension.

We recall here, for the reader convenience, the following results on linearly speciality and effectiveness.

\begin{theorem}[{\cite[Corollary 4.8, Theorem 5.3]{BraDumPos1}}]\label{theorem n+3}
All non-empty linear systems with $s\le n+2$ points  are linearly non-special. 

Moreover, for $s\ge n+3$ let $s(d)\geq0$ is the number of points of multiplicity $d$. If
\begin{equation}\label{EffectivityCondition}
\sum_{i=1}^s m_i\leq nd+  \min (n-s(d),s-n-2), \quad 1\le m_i\le d,
\end{equation}
then $\LL$ is linearly non-special.
\end{theorem}

\begin{theorem}[\cite{BraDumPos1,CDDGP,CT}]\label{old effectivity}
If $s\le n+2$, then $\ls$ is non-empty if and only if 
\begin{equation}\label{effectivity n+2}
m_i\le d, \ \forall i=1\dots, s, \quad \sum_{i=1}^{s}m_i\le nd.
\end{equation}

Moreover if $s\ge n+3$ and \eqref{EffectivityCondition} is satisfied, then $\ls$ is non-empty.
\end{theorem}

Since Theorem \ref{old effectivity} gives a statement about linear systems we assumed that the coefficients $d$ and $m_i$ are positive. In order to translate the statement into the language of divisors on the blow-up of $\PP^n$ at points we relax the positivity assumption on the coefficients $m_i$'s and we obtain the following

\begin{corollary}\label{old effectivity divisors}
The effective cone of $\PP^n$ blown-up at $s\le n+2$ points in general position is described by \eqref{effectivity n+2} and the inequalities \begin{equation}\label{effectivity n+2 for divisors}
d\ge 0, \quad \sum_{i\in I}m_i\le nd,\ \forall I\subseteq\{1,\ldots,s\},\ |I|=n+1.
\end{equation}
\end{corollary}

We remark that \eqref{EffectivityCondition}
is also sufficient condition for the base locus of $\ls$ to not contain any multiple rational normal curve.
In Section \ref{base locus lemma for secants} we will give a sharp base locus lemma for the rational normal curve for all linear
systems based at $n+3$ general points. Moreover 
in Section \ref{effectivity for secants} we will give  necessary and sufficient conditions for a linear system in $\PP^n$ based at $n+3$ general points to be non-empty.

\subsection{Connection to the Fr\"oberg-Iarrobino Conjecture}\label{Froberg-Iarrobino}
The  problem of determining the dimension of linear systems with assigned multiple points is related 
to the Fr\"oberg-Iarrobino Weak and Strong Conjectures \cite{Froberg, Iarrobino}, which give a predicted value for the Hilbert series of
an  ideal generated by $s$ general powers of linear forms  in the polynomial ring with $n+1$ variables.
 Such an ideal corresponds, via apolarity, to the ideal of a collection of fat points, therefore
it is possible to give a geometric interpretation of this conjecture, as Chandler pointed out 
\cite{Chandler}.  See also \cite[Sect. 6.1]{BraDumPos1} for more details.

In terms of our Definition \ref{new-definition} the Weak Conjecture can be stated as follows:
the dimension of 
 a \emph{homogeneous} linear system, i.e. one for which all points have the same multiplicity, 
is bounded below by its linear expected dimension. 

\begin{conjecture}[Weak Fr\"oberg-Iarrobino  Conjecture]\label{WFI}
The linear system $\ls=\ls_{n,d}(m^s)$ satisfies $\dim(\ls)\ge\ldim(\ls)$.
\end{conjecture}

Moreover, the Strong Conjecture states that a homogeneous linear system is always
 linearly  non-special besides a list of exceptions.

\begin{conjecture}[Strong Fr\"oberg-Iarrobino  Conjecture]\label{FI}
The linear system  $\ls=\ls_{n,d}(m^s)$ satisfies $\dim(\ls)=\ldim(\ls)$ except perhaps 
when one of the following conditions holds: 
\begin{enumerate}[(i)]
\item
$s=n+3$; 
\item
$s=n+4$;
\item
$n=2$ and $s=7$ or $s=8$; 
\item 
$n=3$, $s=9$ and {$d\ge 2m$}; 
\item
 $n=4$, $s=14$ and $d=2m$, $m=2$ or $3$.
 \end{enumerate}
\end{conjecture}

{Notice that in \cite{Chandler} the formulation of case \emph{(iv)} requires the condition $d=2m$, 
anyway it is known (see for example \cite{BraDumPos2}) that for any degree $d\ge 2m$ the linear system $\ls_{n,d}(m^s)$ is linearly special for $m$ high enough.
}

{
In this paper we conjecture that linear systems with $n+3$ points are linearly special only if they contain in their base locus the rational normal curve given by the $n+3$ points, its secant varieties 
or their joins with linear subspaces spanned by subsets of the set of the $n+3$ points (Conjecture \ref{conjecture}).}
Moreover, we give a new definition of expected dimension, the \emph{secant linear dimension} $\sldim$ (see Definition \ref{new definition rnc}),
that provides a correction term for  $\ldim$.
In particular, in the homogeneous case this {completes} the Strong Fr\"oberg-Iarrobino  Conjecture.

\begin{remark}
 It would be interesting to 
extract the Hilbert series of ideals generated by $n+3$ powers of linear forms  from our formula of $\sldim$, Definition \ref{new definition rnc}.
\end{remark}

\subsection{General vision}

{We can interpret the base locus lemma for rational normal curves and related cycles 
$-$that we prove in Section \ref{base locus lemma for secants}$-$
and the definition of $\sldim$ both as extensions of the results  
contained
in our previous works \cite{BraDumPos1,DumPos}: Lemma \ref{linear base locus}
and Definition \ref{new-definition}.
}

{We expect rational normal curves and related cycles to appear as obstruction to the non-speciality also in the case of linear systems with $s\ge n+3$ base points.}
A natural generalization of  Conjecture \ref{WFI} would be that 
$\sldim$ provides a lower bound for the dimension
 of any general non-homogeneous linear system.
We plan to further investigate this problem.
 We chose to dedicate this work to the case of $n+3$ points
because it is the first case where non-linear obstructions appear and was not understood 
before for the general case of $\PP^n$.

\section{Secant varieties to rational normal curves and Cremona transformations}
\label{secants}

In this section we collect a series of well-known geometric aspects
of secant varieties to rational normal curves. The first important point is the following.

\begin{theorem}[Veronese]
There exists a unique rational normal curve of degree $n$ passing through 
 $n+3$ points in general position in $\PP^n$.
\end{theorem}
This theorem is classically known and its first proof is due to Veronese
\cite{veronese}, although it is often 
attributed to Castelnuovo.

In this section and throughout this paper we will adopt the following notation. 
Let $p_1,\dots,p_{n+3}\in\PP^n$ general points,  let $C$ be the rational normal 
curve of degree $n$ interpolating them and, for every $t\ge1$, let  
$\sigma_t:=\sigma_t(C)\subset\PP^n$ be
the variety of $t$-secant $\PP^{t-1}$'s to $C$. In this notation we have $\sigma_1=C$. 

Rational normal curves  are never secant defective; in particular we have 
the following formula for the secant dimension:
$$
\dim(\sigma_t)=\min(n,2t-1).
$$
Moreover rational normal curves 
are of minimal secant degree if $2t-1<n$, see \cite{CR}: 
$$\deg(\sigma_t)={{n-t+1}\choose{t}}.$$
Secant varieties are highly singular, in particular for $t\ge2$, $2t-1<n$, we have 
 $\sigma_{t-1}\subset \Sing(\sigma_t)$. 
Moreover the multiplicity of $\sigma_t$ along $\sigma_{\tau}$, for all
$1\le \tau<t$, satisfies the following (see e.g. \cite{CR})):
$$
\mult_{C}(\sigma_t)={{n-t}\choose{t-1}},\quad  
 \quad \mult_{\sigma_{\tau}}(\sigma_t)\ge{{n-t-\tau+1}\choose{t-\tau}}.
$$

\subsection{Cones over the secant varieties to the rational normal curve}
In this section, we consider  cones over the $\sigma_t$ 
with vertex spanned by a subset of the base points. 
Let $I\subset\{1,\dots,n+3\}$ with $|I|=r+1$. We use the conventions $|\emptyset|=0$ and $\sigma_0=\emptyset$.
Let us denote by 
\begin{equation}\label{cones}
\J(L_I,\sigma_t)
\end{equation}
 the join of $L_I$ and $\sigma_t$.

Recall that $\sigma_t=\J(\sigma_{t-1},C)=\J(\sigma_{t-2},\sigma_2)$ etc. 
Notice also that $\J(L_I,\sigma_t)\subset\sigma_{|I|+t}$.

The dimensions of such joins can be easily computed:
\begin{equation}\label{dimension cone}
r_{I,\sigma_t}:=\dim(\J(L_I,\sigma_t))=\dim(L_I)+\dim(\sigma_t)+1=|I|+2t-1.
\end{equation}

\subsection{Divisorial cones}\label{section divisorial cones}

When $\J(L_I,\sigma_t)$ is a hypersurface, namely when $r_{I,\sigma_t}=n-1$ that is  $I$ is such that $|I|=n-2t$,
we can characterize
these cones as the unique section of a certain linear system of hypersurfaces  
of $\PP^n$ interpolating points $p_1,\dots,p_{n+3}$ with multiplicity.

We will denote by $\ls_{n,d}(m_1,\dots,m_s)$ the linear system of degree$-d$
hypersurfaces of $\PP^n$ interpolating the $n+3$ points with multiplicity 
$m_1,\dots,m_s$ respectively.

We first discuss the case when $\sigma_t$ is a hypersurface.
Precisely, when $n=2t$, $I=\emptyset$, we have that $\sigma_t$ is a degree $(t+1)$
 hypersurface with multiplicity $t$ along $C$ 
and 
in particular at the fixed points $p_1,\dots,p_{n+3}$. 
 In this notation we have that $\sigma_t$ belongs to the 
the linear  system $\ls_{2t,t+1}(t^{2t+3})$. Moreover one can prove that it is the only
element satisfying the interpolation condition, see also  Section \ref{multiples secants} (Proposition \ref{multiples of secant}).
For instance for $t=1$ one obtains the plane conic through five points,
 $\ls_{2,2}(1^5)$, for $t=2$ one obtains $\ls_{4,3}(2^7)$. 

\begin{remark}\label{secant is cod 1}
In Section \ref{base locus lemma for secants} (Corollary \ref{secant singular locus}), we will show that, when $n=2t$,
$\sigma_t$ has multiplicity exactly $t-\tau+1$ on
 $\sigma_{\tau}$, for all $1\le \tau<t$.
\end{remark}

Assume now that  $\sigma_t$ has higher codimension in $\PP^n$. 
Fix $I$ such that  $|I|=n-2t\ge1$ and consider 
 $\pi_I:\PP^n \dashrightarrow \PP^{2t}$  the projection from the linear subspace 
$L_I$. Denote by  $C':=\pi_I(C)$ the projection of $C$ and $\sigma'_t:=
\pi_I(\sigma_t(C))$ the projection of its $t$-secant variety. 
Then  $C'$ is a rational normal curve 
of degree $2t$ and $\sigma'_t=\sigma_t(C')$ is the $t$-secant variety to $C'$.
Hence   the hypersurface $\J(L_I,\sigma_t)$ is the cone with vertex  the linear subspace $L_I$ 
over the secant variety $\sigma'_t$.

We  conclude that for any $I$ such that $|I|\ge0$,  the following formula holds:
\begin{equation}\label{cones as divisors}
\J(L_I,\sigma_t)=\ls_{n,t+1}((t+1)^{n-2t},t^{2t+3}).
\end{equation}

We mention that in \cite[Theorem 2.7]{CT} the authors prove that divisors 
of the form \eqref{cones as divisors}
are the rays of the effective cone $\Eff_{\rr}(X)$ (see also Section \ref{movable cone}).

\subsection{The standard Cremona transformation}\label{cremona}
We recall that the \emph{standard Cremona transformation} of $\PP^n$
 is the birational transformation 
defined by the following rational map:
$$
\textrm{Cr}:(x_0:\dots: x_n) \to (x_0^{-1}:\dots: x_n^{-1}),
$$
see  e.g. \cite{Dolgachev}.
Let $\ls=\ls_{n,d}(m_1,\dots,m_s)$ be a linear system based on $s$ points
 in general position;
 we can assume, without loss of generality, that the first $n+1$ are the coordinate points.
The map $\Cr$ can be seen as the morphism associated to the linear system $\ls_{n,n}((n-1)^{n+1})$.
This induces an automorphism of the Picard group of the $n$-dimensional space blown-up at $s$ points
by sending the strict transform of $\ls_{n,d}(m_1,\dots,m_s)$ into the strict transform of 
$$
\textrm{Cr}(\ls):= \ls_{n,d-c}(m_1-c,\dots,m_{n+1}-c, m_{n+2},\dots, m_s)
$$
where 
$$c :=  m_1+\cdots+m_{n+1}-(n-1)d.$$
We have the following equality
 \begin{equation}\label{Cremona preserves dim}
\dim(\ls)=\dim(\textrm{Cr}(\ls)).\end{equation}

If $c\le 0$ we will say that the linear system $\ls$ is \emph{Cremona reduced}.

\begin{remark}\label{cones cremona} One can check that the join divisor  
$\J(L_I,\sigma_t)$ 
is in the orbit of the Weyl 
group of an exceptional divisor. 
To see this, order the multiplicities decreasingly and apply the Cremona action to the 
first $n+1$ points $t+1$ times. 
Indeed,
$$c(\ls_{n,t+1}((t+1)^{n-2t},t^{2t+3}))=1
$$ 
therefore 
$$\Cr(\ls_{n,t+1}((t+1)^{n-2t},t^{2t+3}))=\ls_{n,t}(t^{n-2t+2},(t-1)^{2t+1}).$$
This proves the claim 
since one can recursively replace $t$ by $t-1$ until $t=0$.
\end{remark}

\section{Base locus lemma}\label{base locus lemma for secants}

In this section we give a sharp base locus lemma for the rational normal curve,
 and (cones over)
its secant varieties, that generalizes Lemma \ref{linear base locus}
 from the case of at most $n+2$ points to the case of arbitrary number of points $s$. 

If $s\ge n+3$, as in Section \ref{secants}, we denote by $C$ the unique rational normal curve
through any  subset of $n+3$ points, say $p_1,\dots, p_{n+3}$, and by $\sigma_t$ its $t$-th secant variety.
We denote by $J(I,\sigma_t)$ the join between any index set $I\subset\{1,\dots,s\}$ and $\sigma_t$, and by $r_{I,\sigma_t}$ 
its dimension.
For $s=n+3$, this notions coincide with the ones introduced in \eqref{cones} and \eqref{dimension cone}.

To a linear system $\ls_{n,d}(m_1,\dots,m_{n+3},\dots,m_{s})$ we associate the following integers:
\begin{align}
k_C&:=\sum_{i=1}^{n+3} m_i-nd, \label{multiplicity rnc}\\
k_{I,\sigma_t}&:=\sum_{i\in I}m_i+tk_C-(|I|+t-1)d. \label{multiplicity cone}
\end{align}
Notice that, by setting 
$$
M:=\sum_{i=1}^{n+3}m_i,
$$
one can write 
\begin{equation}\label{expanded}
\ \ \ k_{I,\sigma_t}=tM+\sum_{i\in I}m_i-
{
((n+1)t+|I|-1)d.
}
\end{equation}

Moreover, if  in \eqref{multiplicity cone}
 we replace $t=0$ we obtain 
$$k_I:=\sum_{i\in I} m_{i}-(|I|-1)d\  \quad \quad\quad $$
 (cfr. \eqref{mult k});
 if $|I|=0$ and $t=1$ we obtain $k_C:=k_{\emptyset,\sigma_1}=M-nd$;
if $|I|=0$ we obtain $k_{\sigma_t}:=tk_C-(t-1)d$.

The number $k_{\sigma_t}$ is the multiplicity of containment of a $t$-secant $\PP^{t-1}$ to $C$ in the
base locus of $\ls$. This is a straightforward consequence of the linear base locus lemma, 
knowing that $k_C$ is the multiplicity of containment of $C$.
In the next lemma we prove that in fact the whole $\sigma_t$ is contained in the base locus with
that multiplicity. 

\begin{lemma}[Base locus lemma]\label{base locus lemma rnc}
Let $\ls$ be an effective linear system with $s$ base points. 
In the same notation as above, let $C$ be the rational normal curve given by $n+3$ of them, fix any $I\subset\{1,\dots,s\}$ 
and $t\ge0$ such that $r_{I,\sigma_t}\le n-1$.

If $k_{I,\sigma_t}\ge1$, 
then the cone $\J(L_I,\sigma_t)$ is contained in the base locus with exact 
multiplicity $k_{I,\sigma_t}$.
\end{lemma}

\begin{proof}
Since all of the results used in this proof hold for arbitrary number of points $s$, 
it is enough to prove that statement for $s=n+3$ and for  the corresponding $C$.

If $t=0$ then $\J(L_I,\sigma_t)=L_I$ and the statement follows from 
Lemma \ref{linear base locus}.

Assume that $I=\emptyset$ and $t=1$. Then $\J(L_I,\sigma_t)=C$ is 
the rational normal curve through the $n+3$ points.
In \cite[Theorem 4.1]{carlini-catalisano}, the authors prove that performing 
the Cremona transformation 
based at the first $n+1$ base points of $\ls$, then $C$ is mapped
 to the line through the last two points
 $p_{n+2}$ and $p_{n+3}$, that we may denote by $L_{I(1)}$.  
Let $K_C$ be the multiplicity of containment of $C$ in $\ls$:
 one has $K_C\geq k_C$ by B\'ezout's Theorem.
Observe that $K_C$ is also the multiplicity of containment of the line  $L_{I(1)}$ in  $\Cr(\ls)$. 
We conclude by noticing that by the linear base locus lemma, this is given by 
$$
K_C=k_{I(1)}=m_{n+2}+m_{n+3}-\left(nd-\sum_{i=1}^{n+1}m_i\right)
=M-nd=k_C.
$$

Assume that $I=\emptyset$ and $t\ge2$, $2t-1<n$.
The above parts imply that any secant $(t-1)$-plane spanned by $t$ distinct points of $C$
 is contained in the base locus of $\ls$ with multiplicity exactly $k_{\sigma_{t}}.$
Moreover, since the multiplicity is semi-continuous, it follows that all limits of $t$-secant $(t-1)$-planes 
are contained in the base locus with multiplicity at least $k_{\sigma_{t}}$. Hence the secant variety $\sigma_t$ has 
multiplicity $k_{\sigma_{t}}$.

Finally, the case $I\neq\emptyset$, $t\ge 1$  follows from the above. Indeed every line $L$ in $\J(L_I, \sigma_t)$
 connecting a point of the vertex $L_I$ and a point of the base $\sigma_t$, is contained in the base locus of 
$\ls$ with multiplicity $k_I+k_{\sigma_{t}}-d$.
\end{proof}

\begin{remark}
 Notice  that the effectivity of $\ls$ implies the following inequality 
$k_C\le m_i\le d$, for all $i$. Indeed if $k_C>m_i$ for some $i$, then $\sum_{j\ne i}m_j>nd$, a contradiction by Theorem \ref{effectivity lemma Pn n+3}. 
This in particular implies  $k_{\sigma_{|I|+t}}\le k_{I,\sigma_t}\le k_I$. 
Moreover the obvious equality $k_{I,\sigma_t}=\sum_{i\in I}m_i-|I|d+k_{\sigma_t}$ and 
the effectivity condition $m_i\le d$ imply  $k_{I,\sigma_t}\le k_{\sigma_t}$.

Because of the containment relations
 $L_I, \sigma_t\subseteq \J(L_I,\sigma_t)\subseteq\sigma_{|I|+t}$, 
the above inequalities read as: if $L_I$ or $\sigma_t$ is not in the base locus of $\ls$,  
neither is $\J(L_I,\sigma_t)$ nor $\sigma_{|I|+t}$; if $\J(L_I,\sigma_t)$ is not contained in the base locus, neither is $\sigma_{|I|+t}$.
\end{remark}

\subsection{Geometric consequences of the base locus lemma}

An immediate consequence of the base locus lemma is a description of the
singularities of the secant variety, whenever this is a hypersurfaces. Indeed since
$\sigma_t$ is the unique element of the linear system $\ls_{2t,t+1}(t^{2t+3})$,
 one can compute  the
multiplicity along the lower order secant varieties.
\begin{corollary}\label{secant singular locus}
 Let $n=2t$ and $1\le \tau\le t$. Then $\sigma_t$ is singular with multiplicity 
$t-\tau+1$ on $\sigma_{\tau}\setminus\sigma_{\tau-1}$
\end{corollary}

Another consequence of Lemma \ref{base locus lemma rnc} is the following result that in particular implies that Cremona reduced 
linear systems are movable.

\begin{corollary}\label{cremona reduction implies movable}
Let $\ls$ be an effective linear system with arbitrary number of points.
Assume that $\ls$ is Cremona reduced. 
Then $\ls$ does not contain any divisorial component of type $\J(I(n-2t-1),\sigma_t)$ in its base locus.
\end{corollary}
\begin{proof}
Write $n=2l+\epsilon$, with $\epsilon\in\{0,1\}$.
By Lemma \ref{base locus lemma rnc}, it is enough to prove that 
$k_{I(n-2t-1),\sigma_t}\le0$ for all $0\le t\le l+\epsilon$.

Since $\ls$ is Cremona reduced, the hyperplane spanned by the collection 
of points parametrized by $I(n-1)$ is not contained in the base locus, for any $I(n-1)$.
 Indeed if $I(n-1)\subset I(n)$, for some $I(n)$, we have 
$$
k_{I(n-1)}<\sum_{i\in I(n)}m_i-(n-1)d\le 0.
$$
This proves the statement for $t=0$.

Assume $1\le t\le l+\epsilon$. For any fixed index set $I:=I(n-2t-1)$ 
of cardinality $n-2t$, choose $2t$ distinct indices in its complement:
$\{{i_1},\dots,{i_{2t}}\}\subset \{1,\dots,n+3\}\setminus I$.
We have
\begin{align*}
k_{I,\sigma_t}&=tM+\sum_{i\in I}m_i-(t+1)(n-1)d\\
&=\sum_{j=1}^s\left(M-m_{i_j}-m_{i_{j+t}}-(n-1)d\right)
+\left(\sum_{j=1}^{2t}m_{i_j}+\sum_{i\in I}m_i-(n-1)d\right)<0
\end{align*}
The first $t$ terms are negative by assumption, the last is strictly negative because of  the hyperplane case $t=0$.
\end{proof} 

\begin{remark}
In Section \ref{movable cone} we will see that effective divisors in the blown-up 
$\PP^n$ at $n+3$ general points without fixed components 
of type $\J(I(n-2t-1),\sigma_t)$, namely those satisfying $k_{I(n-2t-1),\sigma_t}\le0$,
 are movable. Hence the cone of Cremona reduced effective divisors, that is polyhedral since 
defined by inequalities, is contained in the movable cone, that is in turn contained in the effective cone.
\end{remark}

\section{Effective and movable cones}\label{effectivity for secants}

In this section we will give necessary and sufficient conditions for 
 linear systems $\ls_{n,d}(m_1,\dots,m_{n+3})$ in $\PP^n$ with $n+3$ base points in general position to have at least one section.
This is equivalent to an
 \emph{effectivity theorem} for divisors on $X$,  the blown-up $\PP^n$ at $n+3$ 
general points, and
 provides a generalization of Theorem \ref{old effectivity} and Corollary \ref{old effectivity divisors}.

Throughout this section, we will use the same notation introduced in 
Section \ref{base locus lemma for secants}. Moreover we let 
$\NS^1(X)$ denote the Neron-Severi group of $X$ with coordinates 
$(d,m_1,\dots,m_{n+3})$
 corresponding to the hyperplane class and the 
classes of the exceptional divisors.

\begin{theorem}[Effectivity Theorem]\label{effectivity lemma Pn n+3}
\begin{enumerate}
\item[(I)] For $n\geq 2$, a linear system $\ls=\ls_{n,d}(m_1,\dots,m_{n+3})$ is non-empty if and only if 
\begin{align*}
(A_n)&   &  m_i\le &\ d,&\forall i=1,\dots,n+3,\\
(B_n)& &  M-m_i\le&\ nd, &\forall i=1,\dots,n+3,\\
(C_{n,t})& &   k_{I,\sigma_t}\le&\ 0, & \forall   |I|=n-2t+1, \ 1\le t\le l+\epsilon,
\end{align*}
where $n=2l+\epsilon$, $\epsilon\in\{0,1\}$.

\item[(II)]
The facets of the effective cone of divisors $\Eff_{\rr}(X)$ are given by the equalities in 
$(A_n)$, $(B_n)$, and $(C_{n,t})$, with $-1\le t\le l+\epsilon$.
\end{enumerate}
\end{theorem}

\begin{remark} 
In Section \ref{base locus lemma for secants} the numbers $k_{I,\sigma_t}$ \eqref{multiplicity cone} appeared as the
 multiplicities of containment 
in the base locus, of cycles of codimension at least one, namely for $1\le t\le l+\epsilon$ and $0\le |I|\le n-2t-1$.
In this section, the numbers $k_{I(n-2t),\sigma_t}$ appearing in $(C_{n,t})$ (Theorem \ref{effectivity lemma Pn n+3}) 
are formal generalizations of the above to the zero codimensional case: $\J(I(n-2t),\sigma_t)=\PP^n$.
In analogy with the cases $s= n+1, n+2$ points, where the condition $\sum_{i=1}^sm_i\le nd$ 
corresponds to asking that the linear span of all of the points $-$ that is the whole space $\PP^n$ $-$ is not in the base locus
(see \cite[Theorem 1.6]{DumPos}),
here we ask that no $n$-dimensional \emph{virtual} cycle $\J(I(n-2t),\sigma_t)$ is in the base locus of linear systems with $n+3$ points. 
\end{remark}

The proof of Theorem \ref{effectivity lemma Pn n+3} is by induction on $n$. For this reason we think it is convenient to treat the initial case, $n=2$, separately. Here the conditions of part \emph{(I)} read as follows.

\begin{align*}
(A_2)&   &  m_i\le &\ d,&\forall i=1,\dots,5,\\
(B_2)& &  M-m_i\le& \ 2d, &\forall i=1,\dots,5,\\
(C_{2,1})& &   M+m_i\le&\ 3d, & \forall i=1,\dots,5,
\end{align*}

\begin{proof}[Proof of Theorem \ref{effectivity lemma Pn n+3}, part (I), case $n=2$]
Without loss of generality we may assume $m_1\ge m_2\ge\cdots\ge m_5\ge1$.
It is enough to prove that $\ls$ is non-empty if and only if
$$
m_1\le d, \quad 
m_1+m_2+m_3+m_4\le 2d, \quad 
2m_1+m_2+m_3+m_4+m_5\le 3d.
$$

\medskip
If $\ls$ is non-empty, then obviously $m_1\le d$. Moreover since  $\ls_{2,d}(m_1,\dots,m_4)$ is non-empty, then $m_1+m_2+m_3+m_4\le 2d$.

To prove the third inequality assume first that $k_C=\sum_{j=1}m_j-2d\le0$. Then 
$2m_1+m_2+m_3+m_4+m_5-3d= k_C+m_1-d\le0$. If $k_C\ge 1$, then 
by Lemma \ref{base locus lemma rnc}, the conic $C$ is a fixed component of $\ls$ 
and the residual part has degree $d':=d-2k_C=5d-2\sum_{j=1}^5m_j$ 
and multiplicities $m'_i=2d-\sum_{j=1}^5m_j+m_i$ at the $5$ points. Notice that $m'_i\ge0$ by the second inequality.
Effectivity implies $d'\ge m'_1$ and this is equivalent
to the third inequality.

\medskip
We now prove the other implication.
If $k_C=\sum_{j=1}^5m_j-2d\le0$ then $\ls$ is non-empty by Theorem \ref{old effectivity}. 
Assume $k_C\ge1$.
By Lemma \ref{base locus lemma rnc}  the conic through the five points 
is contained in the base locus with multiplicity $k_C$. 
Notice that $m_1+m_2+m_3+m_4\le 2d$ implies $k_C\le m_5$. 
The residual is $\ls'$ with $d'=d-2k_C=5d-2\sum_{j=1}^5m_j$ and $m'_i=m_i-k_C=2d-\sum_{j=1}^5m_j+m_i\ge0$, for all $i=1,\dots,5$.
Obviously $\sum_{j=1}^5m'_j-2d'=0$. We claim $d'\ge m'_i$, for all $i=1,\dots,5$. Hence $\ls'$ is effective. 
To prove the claim for $i=1$, notice that $d'-m'_1=5d-2\sum_{j=1}^5m_j -2d+\sum_{j=1}^5m_j-m_1=3d-\sum_{j=1}^5m_j-m_1\ge0$.
\end{proof}

We now complete the proof of the first part of effectivity theorem for $n\ge 3$.

\begin{proof}[Proof of Theorem \ref{effectivity lemma Pn n+3}, part (I), $n\ge3$]
Without loss of generality, we may reorder the points so that $m_1\ge\cdots\ge m_{n+3}$.
If $m_{n+3}=0$, the set of conditions $(B_n)$ 
becomes just $\sum_{j=1}^{n+2}m_j\le nd$ and the third set of conditions, $(C_{n,t})$ is redundant. 
In this case the result was proved in \cite[Lemma 2.2]{BraDumPos1}, see Theorem \ref{old effectivity}.
Hence we will assume $m_{n+3}\ge1$.

\medskip\medskip
\emph{``Only if'' implication:}

If $\ls$ is effective then $(A_{n})$ and $(B_{n})$ trivially hold. 

The {expanded} expressions of condition $(C_{n,t})$ is 
$$
k_{I,\sigma_t}=tM+\sum_{i\in I}m_i-((t+1)n-t)d\le 0, 
$$ 
for all $1\le t\le \left\lfloor\frac{n-1}{2}\right\rfloor=l+\epsilon$  and all multi-index $I=I(n-2t)$, see \eqref{expanded}.
Fix such $t$ and $I=I(n-2t)$. Take any $j\in I(n-2t)$ and denote by $I\setminus\{j\}$ its complement in $I$.
Write
 $$k_{I\setminus\{j\},\sigma_t}=tM+\sum_{i\in I\setminus\{j\}}m_i-((t+1)n-(t+1))d.$$
 In order to prove the inequality $(C_{n,t})$, we consider the following cases.

\medskip
Case (1). Assume that $k_{I\setminus\{j\},\sigma_t}\le0$.
Since $k_{I,\sigma_t}=k_{I\setminus\{j\},\sigma_t}+(m_j-d)$, we conclude by $(A_n)$.

\medskip
Case (2).  Assume that $k_{I\setminus\{j\},\sigma_t}\ge1$. By Lemma \ref{base locus lemma rnc}, 
$k_{I\setminus\{j\},\sigma_t}$ is the multiplicity of containment of the cone 
$\J(I\setminus\{j\},\sigma_t)$, that is the hypersurface $\ls_{n,t+1}((t+1)^{n-2t},t^{2t+3})$, see \eqref{cones as divisors}.
 The residual of $\ls$ after its removal is still effective by assumption. 
Let us denote by $d'$ and by $m'_j$ the degree and the multiplicity at the point 
$p_j$ of the residual, that is $d'=d-(t+1)k_{I\setminus\{j\},\sigma_t}$  and 
 $m'_j=m_j-tk_{I\setminus\{j\},\sigma_t}$. Effectivity implies that 
$d'\ge m'_j\ge0$. We conclude by noticing that $d'-m'_j\ge 0$  is equivalent to 
$k_{I,\sigma_t}\le0$.

\medskip\medskip

\emph{``If'' implication:}

The proof is by induction on $n$, with initial step 
the case $n=2$ for which the statement is already proved to hold.
Assume the statement true for $n-1$.

We will construct recursively an element that belongs to $\ls$, hence
proving non-emptiness. We will treat the following  cases and subcases separately.
\begin{itemize}
\item[(0)] $k_C\le 0$.
\item[(1)] $k_C\ge 1$ and $m_1=d$.
\item[(2)] $k_C\ge 1$ and $m_1\le d-1$,
\begin{itemize}
\item[(2.a)] $\ls$ is Cremona reduced,
\item[(2.b)] $\ls$ is not Cremona reduced.
\end{itemize}
\end{itemize}

\medskip
Case (0). In this case $\ls$ is effective by Theorem \ref{old effectivity}. 

\medskip
Case (1). 
Notice that the elements of $\ls$, cones with 
vertex at $p_1$, are in bijection with  the elements of a linear system 
$\ls'=\ls_{n-1,d}(m_2,\dots,m_{n+3})$.  
One can check easily that $\ls'$ satisfies conditions $(A_{n-1})$, $(B_{n-1})$, being those implied by 
 $(A_{n})$, $(B_{n})$  respectively. Moreover
for $1\le t\le l$ and any index set 
 $I=I(n-2t)$ such that  $1\in I$, 
condition $(C_{n,t})$ implies condition
$(C_{n-1,t})$, for the index set $I'=I\setminus\{1\}=I(n-1-2t)$. Indeed we have
\begin{align*}
k_{I',\sigma_t}&=tM'+\sum_{i\in I'}m'_i-((t+1)(n-1)-t)d'\\
&=tM+\sum_{i\in I}m_i-(t+1)m_1-
((t+1)n-t)d+(t+1)d\\
&=k_{I,\sigma_t}\le0.
\end{align*}
Since $\ls'$ is effective, then $\ls$ is.

\medskip
Case (2.a).
Notice that in this case $m_{n+3}\ge2$. 
Set  $I:=\{1,\dots, n-2\}$ and consider the cone $\J(I,C)$ over the rational normal curve $C$
with vertex the linear subspaces $L_I$ spanned by the first $n-2$ points. 
As in \eqref{cones as divisors}, 
$\J(I,C)$ can be interpreted as the fixed divisor $\ls_{n,2}(2^{n-2},1^5)$. 
Let us denote by $\ls'$ the kernel of the restriction map $\ls\to \ls|_{\J(I,C)}$. We can write
$$\ls'=\ls_{n,d'}(m'_1,\dots,m'_{n+3}):=\ls_{n,d-2}(m_1-2,\dots,m_{n-2}-2,m_{n-1}-1,\dots, m_{n+3}-1),$$
 and the inequality $\dim(\ls)\ge\dim(\ls')$ is satisfied.

Notice that if $m_{n-2}=2$, i.e. $m'_{n-2}=0$, then $\ls'$ is based on at most $n+2$ points. 
In this case we have 
\begin{align*}
\sum_{i=1}^{n+3}m'_i-nd'&=(M-2s-1)-n(d-2)\\
& =(M-m_1-m_{n-2}-(n-1)d)+(m_{n-2}+m_1-d-1)\le 0.
\end{align*}
The inequality follows
by the fact that $\ls$ is Cremona reduced and that $m_1\le d-1$.
One concludes by noticing that  $\ls'$ falls into case (0).

Otherwise, if $m_{n-2}\ge3$, since we also have $m_{n+3}\ge2$, i.e. $m'_{n+3}\ge1$, then $\ls'$ is based 
on $n+3$ points. We claim that  such a $\ls'$ satisfies conditions
 $(A_n)$, $(B_n)$, $(C_{n,t})$. 
Moreover $k'_C=k_C-1$, namely the multiplicity of containment of $C$ in the base locus
of $\ls'$ has decreased by one.
If $k'_C=0$ we conclude by case (0), otherwise we proceed with cases (1) or (2).

We are now left with showing the claim.
One can easily check that the first two conditions are satisfies, because of 
the assumption $m_i\le d-1$.
In order to prove that the third set of conditions, $(C_{n,t})$, is also satisfied 
for any set $I=I(n-2t)$, 
$n\ge 2t-1$, notice that 
$$\sum_{i\in I} m'_i \le \sum_{i\in I} m_i-(n-2t+1+f),$$
where $f$ is the cardinality of the index set $I\cap\{1,\dots,n-2\}$.
From this,  it follows that
$$
tM'+\sum_{i\in I} m'_i-((t+1)n-t)d'\le tM+\sum_{i\in I} m_i-((t+1)n-t)d+(n-t-f-1).
$$
Now, choose $2t$ distinct indices 
$\{{i_1},\dots,{i_{2t}}\}\subset \{1,\dots,n+3\}\setminus I$;
the right hand side of the above expression equals
$$
\sum_{j=1}^t\left(M-m_{i_j}-m_{i_{j+t}}-(n-1)d\right)
+\alpha,
$$
where
$$
\alpha:= \left(\sum_{j=1}^{2t}m_{i_i}+\sum_{i\in I}m_i-nd\right)+(n-t-f-1).
$$
Here we introduce the integer $\alpha$ for the sake of simplicity as we will treat different cases in what follows. 
Notice that  because of the assumption that $\ls$ is Cremona reduced, in order to conclude 
it is enough to prove that $\alpha\le 0$.

Assume $d\ge n-t-1$. We have 
$$
\alpha
=\left(\sum_{j=1}^{2t}m_{i_j}+\sum_{i\in I}m_i-(n-1)d\right)
+(n-t-f-1-d)\le0,
$$
where the inequality follows from the fact that $\ls$ is Cremona reduced.

If $d\le n-t-2$, using $m_i\le d-1$ we obtain
$$
 \alpha
\le (n+1)(d-1)-nd+(n-t-f-1)=n-2t-4-f\le0,
$$
where the last inequality {is implied} by the fact that $f\ge\min\{0,n-2t-4\}$. 

\medskip
Case (2.b). Assume that $\ls$ is not Cremona reduced, namely that 
$$c:=\sum_{i=1}^{n+1}m_i-(n-1)d\ge1$$
and write
$$\ls':=\Cr(\ls)=\ls_{n,d'}(m'_1,\dots,m'_{n+3}).$$
We have  $\dim(\ls)=\dim(\ls')$ by \eqref{Cremona preserves dim}.
We claim that $\ls'$
satisfies conditions $(A_n)$, $(B_n)$, $(C_{n,t})$.
Hence we can reiterate the entire procedure for $\ls'$, hence reducing the proof of the effectivity
of $\ls$ to the proof of the effectivity of its Cremona transform $\ls'$.

We now prove the claim. 
We refer to conditions  $(A_n)$,  $(B_n)$  and  $(C_{n,t})$ for $\ls'$ 
as $(A_n)'$,  $(B_n)'$  and  $(C_{n,t})'$.

 Notice that  $m'_i\le d'$ if and only if $m_i\le d$, for all $i\le n+1$. Moreover $m'_i\le d'$
 is equivalent to $(B_n)$, for $i=n+2,n+3$. 

One can easily check that $(B_n)$ implies $(B_n)'$.

 We now prove that $(C_{n,t})'$ is satisfied for any $t$ and any index set $I=I(n-2t)$, with $n\ge 2t-1$.
The expanded expression of condition $(C_{n,t})'$ is 
$$
tM'+\sum_{i\in I} m'_i-((t+1)n-t)d'\le0.
$$
Assume that $I\subset\{1,\dots,n+1\}$. The left-hand side of the above expression equals
$$
tM+\sum_{i\in I} m_i-((t+1)n-t)d-c
$$
that is negative because $c\ge1$ and  $(C_{n,t})$ is satisfied.

Assume  that $| I\setminus\{1,\dots,n+1\}|=1$, that {occurs} only if $n\ge 2t$. 
One can easily verify that $(C_{n,t})'$ is equivalent to  $(C_{n,t})$.

Assume that $| I\setminus\{1,\dots,n+1\}|=2$, that occurs only when $n\ge 2t+1$.
The left-hand side of the expanded expression of condition $(C_{n,t})'$ equals
$$
tM+\sum_{i\in I} m_i-((t+1)n-t)d+c
=(t+1)M+\sum_{i\in I(n-2t-2)} m_i-((t+2)n-(t+1))d
$$
and this is bounded above by zero by $(C_{n,t+1})$.
\end{proof}

\begin{proof}[Proof of Theorem \ref{effectivity lemma Pn n+3}, part (II)]
Notice that the arithmetic condition $(C_{n,-1})$ corresponds to  $d\geq 0$ and $(C_{n,0})$ corresponds to  $M-m_i-m_j\leq nd$. 
When multiplicities $m_i$ are positive, condition $(C_{n,0})$ is redundant and the statement was proved in part \emph{(I)}. 

We assume now $m_1\ge \cdots\ge m_{n+3}$ and  $m_{n+3}<0$. 
Then conditions $(C_{n,-1})$ and $(C_{n,0})$ imply $(C_{n,t})$ for $t\geq 1$. 
Moreover the non-emptiness of $\ls$ is equivalent to the non-emptiness of $\ls_{n,d}(m_1,\dots,m_{n+2})$.
This holds because every effective divisor $D$ in $\ls$ decomposes as $D=-m_{n+3}E_{n+3}+(D+m_{n+3}E_{n+3}).$ The statement follows now from the description of the effective cone of divisors on the blown-up $\PP^n$ at $s\leq n+2$ points, that is given  by $(A_n)$, $(B_n)$, $(C_{n,-1})$ and $(C_{n,0})$, see Corollary \ref{old effectivity divisors}.
\end{proof}

\subsection{Movable Cone of Divisors}\label{movable cone}
{\emph{Mori dream spaces}} were introduced by Hu and Keel (\cite{HuKeel}, see also \cite{ADHL}), we give now an alternative definition. 
Let $X$ be a normal $\QQ$-factorial variety whose Picard group, $\Pic(X)$, is a lattice. Define the \emph{Cox ring} of X as
$$\Cox(X):=\bigoplus_{D\in \Pic(X)} H^0(X, D)$$
with multiplicative structure defined by a choice of divisors whose classes form a 
basis for the Picard group $\Pic(X)$.
We say that $X$ is a Mori dream space if the Cox ring, $\Cox(X)$, is finitely generated.

We define the \emph{movable cone} of a variety, $\Mov(X)$, to be the cone generated by 
divisors without divisorial base locus. 
For a Mori dream space, the movable cone and the effective cone of divisors are polyhedral. Moreover, 
the movable cone decomposes into disjoint union of \emph{nef chambers}  that are the nef cones of all small 
$\QQ$-factorial modifications. 

Let $X$ now denote   the blow-up of the projective space $\PP^n$
at $s\leq n+3$ points in general position. We first recall that $X$ is a 
Mori dream space \cite{CT} (see also  \cite{Mukai2}). Moreover results of \cite{CDDGP,CT} imply 
that the effective cone $\Eff_{\rr}(X)$ is generated as a cone and semigroup by divisors in the
 \emph{Weyl orbit} $W\cdot E_i$. Here, $W$ represents the \emph{Weyl group} of $X$ and $W\cdot E_i$
 the orbit with respect to 
 its action on an exceptional divisor. Recall that every element of $W$ corresponds to a birational 
map of $\PP^{n}$ lying in the group generated by 
projective automorphisms and standard  Cremona transformations of $\PP^{n}$.
Also, $$\Pic(X)=\langle H, E_1, E_2, \ldots, E_{s}\rangle.$$ Following \cite{Mukai} we introduce a symmetric bilinear on
 $\Pic(X)$ acting on its generators as:
\begin{equation}\label{bilinear form}
\langle E_{i}, E_{j}\rangle=-\delta_{i,j}, \ \langle E_{i}, H\rangle=0, \ \langle H,
H\rangle= n-1.
\end{equation}

Let $\Eff_{\rr}(X)^\vee$ denote the dual cone of the cone of effective divisors with real coefficients. Namely, the dual cone 
$\Eff_{\rr}(X)^\vee$ consists of divisors $D$ such that $\langle D,F \rangle \geq 0$ with all $F\in \Eff_{\rr}(X).$
One can define  the \emph{degree} of a divisor $D\in \Pic(X)$ as follows (see \cite{CT}):
$$\deg(D):=\frac{1}{n-1}\langle D, -K_{X} \rangle,$$
where $K_{X}$ denotes the canonical divisor of the blown-up projective space, $X$.
For $s\leq n+3$ the movable cone can be described as the intersection 
between the cone of effective divisors with real coefficients and its dual
\begin{equation}\label{movable as intersection}
\Mov(X) = \Eff_{\rr}(X) \cap \Eff_{\rr}(X)^\vee,
\end{equation}
 (see \cite[Theorem 4.7]{CDDGP}). 

The facets of the effective cone $\Eff_{\rr}(X)$ for $s\le n+2$ 
are given in Corollary \ref{old effectivity divisors}. 
As an application of this result and using \eqref{movable as intersection}, one can easily describe 
the facets of 
the movable cone $\Mov(X)$ for $s\leq n+2$ 
(see \cite{CDDGP}). 

We will extend this description to the case with $s=n+3$ points.
The facets of the effective cone $\Eff_{\rr}(X)$ of the blown-up $\PP^n$ at
$n+3$ points in general position are computed in Theorem
\ref{effectivity lemma Pn n+3}.
Moreover the generators of the effective cone  are described in \cite{CT} by 
the classes of divisors in the set
\begin{equation}\label{generating set A}
\mathcal{A}=\left\{(t+1)H-(t+1)\sum_{i\in I}E_i-t\sum_{i\notin I}E_i:  |I|=n-2t, 
-1\le t\le l+\epsilon\right\},
\end{equation}
where $n=2l+\epsilon$, $\epsilon\in\{0,1\}$. 
Notice that the divisors in $\mathcal{A}$ are the only divisors on $X$ of degree 1 and are the strict 
transforms of the one-section linear systems 
 described  in 
\eqref{cones as divisors};
these are precisely the divisors
 that will appear in Conjecture \ref{conjecture}.

We can now describe the movable cone of divisors on the blown-up $\PP^n$
 at the collection of $n+3$ points. 
They are the effective divisors on $X$ for which the corresponding linear system
has no fixed divisorial component of type $\mathcal{A}$.

\begin{theorem}\label{thm movable cone}
 For $n\geq 2$, letting $(d, m_1, \ldots, m_{n+3})$ 
 be the coordinates of the Neron-Severi group, 
$\NS^1(X)$, then the movable cone $\Mov(X)$ is generated by the inequalities $(A_n)$ and 
$(B_n)$,  of Theorem \ref{effectivity lemma Pn n+3} and
\begin{align*}
(D_{n,t})\quad\quad \quad \quad
k_{I,\sigma_t}\le0,  \quad\quad\quad \forall |I|=n-2t, \ -1\le t\le l+\epsilon.
\end{align*}
\end{theorem}

\begin{proof}
It follows from Theorem \ref{effectivity lemma Pn n+3}, \cite[Theorem 4.7]{CDDGP} and 
\cite[Theorem 2.7]{CT}. 
 Indeed a divisor in $\NS^1(X)$ of the form 
$$
D=dH-\sum_{i=1}^{n+3}m_i E_i
$$
with $d,m_i\ge 0$
lies in $\Eff_{\rr}(X)^\vee$ if and only if it has non-negative intersection 
number \eqref{bilinear form} with all elements of the generating
 set $\mathcal{A}$ described in
\eqref{generating set A}. We leave it to the reader to
verify that these conditions are equivalent to
the set of inequalities $(D_{n,t})$. Notice also that the conditions $(C_{n,t})$ in Theorem \ref{effectivity lemma Pn n+3} are redundant. One concludes the proof by using  \eqref{movable as intersection}.  
\end{proof}

\subsection{Faces of the movable cone and contractions.}
From Mori theory it follows that the faces of the movable cone are in 
one to one correspondence with classes of divisorial and fibre 
type contractions from small $\QQ$-factorial modifications of 
$X$ to normal projective varieties.

In particular, contractions given by divisors in the boundary of the effective cone, 
corresponding to the first three sets of equalities, namely $(A_n)$ and $(B_n)$,
 are of fibre type contractions (i.e. projections to lower dimensional
 Mori dream spaces), while contractions associated to the last set of equalities, 
namely $(D_{n,t})$,
 corresponding to the boundary of the dual
 effective cone, are divisorial contractions.

\section{A new notion of expected dimension}\label{new exp dim with secants}

Secant varieties {to the rational normal curve} and cones over them are a natural generalization of the linear obstructions.
In this section we introduce a new notion of expected dimension  for linear systems with $n+3$ points in 
general position, Definition \ref{new definition rnc},
that takes into account their contributions.
Furthermore we conjecture that those are
the {only non-linear obstructions}, see Conjecture \ref{conjecture}.

In Section \ref{evidences} we prove this conjecture for $n\le 3$ and for 
some homogeneous linear systems in families. 

We adopt the same notation as in the previous sections \eqref{dimension cone} and
\eqref{multiplicity cone}.
We recall here that the join $\J(I, \sigma_t)$ has dimension $r_{I,\sigma_t}\le n-1$ whenever
$0\le t \le l+\epsilon$, $n=2l+\epsilon$ and $0\le|I|\le n-2t$.

\begin{definition}\label{new definition rnc}
Let $\LL=\LL_{n,d}(m_1,\ldots,m_{n+3})$ be a linear system. 
The {\em (affine) secant linear virtual dimension} of $\LL$ is the number
\begin{equation}\label{RNC expected dimension}
\sum_{I,\sigma_t}(-1)^{|I|}{{n+k_{I,\sigma_t}-r_{I,\sigma_t}-1}\choose n},
\end{equation}

where the sum ranges over all indexes $I\subset\{1,\dots,n+3\}$ and $t$ such that 
$0\le t \le l+\epsilon$, $n=2l+\epsilon$  and $0\le|I|\le n-2t$.

The {\em (affine) secant linear expected dimension} of $\LL$, denoted by $\sldim(\LL)$ is 
defined as follows: if
the linear system $\LL$ is contained in a linear system whose secant linear virtual dimension
 is negative, then
we set $\sldim(\LL)=0$, otherwise we define $\sldim(\LL)$ to be
 the maximum between the secant linear virtual dimension of $\LL$ and $0$.
\end{definition}

\begin{remark}
Using the base locus lemma (Lemma \ref{base locus lemma for secants}), one may generalise formula \eqref{RNC expected dimension}
for arbitrary  number of points, by taking into account all of the rational normal curves
of degree $n$ given by sets of $n+3$ points  (and related cycles).
\end{remark}

\begin{remark}\label{no multiple rnc}
One can easily verify that $k_{I,\sigma_t}\le r_{I,\sigma_t}$, so its corresponding Newton binomial 
in \eqref{RNC expected dimension} is zero. In particular one can check that this inequality is satisfied for all $I$ and $t$ when $k_C\le 1$, namely
when the rational normal curve $C$ is contained in the base locus of $\ls$ at most simply.
In all of these cases  $\sldim(\ls)=\ldim(\ls)$.
\end{remark}

\begin{conjecture}\label{conjecture}
Let $\LL$ be a non-empty linear system of $\PP^n$ with $n+3$ base points in general position
 and let $C$ be the rational normal curve through the base points. 
{Then $\LL$ is special only if its base locus contains 
either linear cycles, or cones over the secant varieties $\sigma_t$ of $C$.
Moreover, we have
$\dim(\LL)=\sldim(\LL).$}
\end{conjecture}

We illustrate this idea in the following examples.
\begin{example}
The linear system $\LL=\LL_{6,8}(6^9)$ is linearly special, since
$\dim(\LL)=1$ and $\ldim(\LL)=-147$. The rational curve  $C$, given by the $9$ base points, 
is contained in the singular locus of the fixed hypersurface $\ls$
 with multiplicity $k_C=6$.
Moreover, for each of the $9$ base points, say $p$, the 
cone $\J(p,C)$ as well as $\sigma_2$
are contained with multiplicity $4$ in the singular locus of $\ls$, by Lemma \ref{base locus lemma rnc}.
Hence one can compute $\sldim=1$.
\end{example}

\begin{example}
Consider the linear system $\ls_{4,10}(9,7^3,5^3)$.
The rational normal curve is contained $5$ times and the cone $J(p_1,C)$ is contained with multiplicity $4$ 
 We leave it to the reader to verify that
$\sldim(\ls)=2$ and that $\dim(\ls)=2$, the last equality following a series of Cremona transformations (see Section \ref{cremona}).
\end{example}

\subsection{Properties of $\sldim$}

In this section we prove 
two technical lemma which will be useful in the sequel.

Recall that a linear system $\LL=\LL_{n,d}(d,m_2,\ldots,m_s)$ has the same dimension of the linear system 
$\LL_{n-1,d}(m_2,\ldots,m_s)$. We will call the second system the {\it cone reduction} of $\LL$ and we denote it by 
 $\textrm{Cone}(\LL)$.

\begin{lemma}\label{cone-invariant}
The secant linear expected dimension of a linear system $\ls=\LL_{n,d}(d,m_2,\ldots,m_{n+3})$ is invariant under cone reduction:
$$\sldim(\LL)=\sldim(\mathrm{Cone}(\LL)).$$
\end{lemma}
\begin{proof}
Let $\textrm{Cone}(\LL)=\LL_{n-1,d}(m_2,\ldots,m_{n+3})$ be the cone reduction of $\ls$.

We write the formula \eqref{new definition rnc} for $\sldim(\LL)$ as follows:
$$\sldim(\LL)=\sum_{I,t}B_\ls(\J(L_I,\sigma_t))$$
denoting by $B_\ls(\J(L_I,\sigma_t))$ the contribution in the sum given by the cycle $\J(L_I,\sigma_t)$
that is $$B_\ls(\J(L_I,\sigma_t)):=(-1)^{|I|}\binom{n+k_{I,\sigma_t}-r_{I,\sigma_t}-1}{n}.$$

Now, recalling the formula ${{a}\choose{b}}=\binom{a-1}{b}+\binom{a-1}{b-1}$,
it is easy to check that,
for any $I\subseteq\{2,\ldots,n+3\}$, one has
$$B_\LL(\J(L_I,\sigma_t))+B_\LL(\J(L_{I\cup\{1\}},\sigma_t))=B_{\textrm{Cone}(\LL)}(\J(L_I,\sigma_t)).$$
\end{proof}

\begin{lemma}\label{cremona-invariant}
The secant linear expected dimension of a linear system $\ls=\LL_{n,d}(m_1,\ldots,m_{n+3})$ is invariant under Cremona transformations:
$$\sldim(\LL)=\sldim(\Cr(\LL)).$$
\end{lemma}

\begin{proof}
Assume that
$$c=\sum_{i=1}^{n+1}m_i-(n-1)d\ge1$$
and let
$$\Cr(\LL)=\LL_{n,d-c}(m_1-c,\ldots,m_{n+1}-c,m_{n+2},m_{n+3}).$$ 
be the system obtained after the Cremona transformation, see Section \ref{cremona}.

First of all consider the linear system obtained from $\LL$ forgetting the last two points:
$\widetilde{\LL}=\LL_{n,d}(m_1,\ldots,m_{n+1})$ and 
let $\Cr(\widetilde{\LL})=\LL_{n,d-c}(m_1-c,\ldots,m_{n+1}-c)$ be the corresponding Cremona transform.
Since $\widetilde{\LL}$ and $\Cr(\widetilde{\LL})$ are linearly non-special by Theorem \ref{theorem n+3}
and a Cremona transformation preserves the dimension of a linear system (see \eqref{Cremona preserves dim}) we have:
$$\ldim(\widetilde{\LL})=\dim(\widetilde{\LL})=\dim(\Cr(\widetilde{\LL}))=\ldim(\Cr(\widetilde{\LL})).$$
Using the same notation as in the proof of Lemma \ref{cone-invariant}, one can split the sum as follows:
\begin{align*}
\sldim(\LL)=\ldim(\widetilde{\LL})&
+{\sum_{|I\cap\{n+2,n+3\}|=1}B_\LL(\J(L_I,\sigma_t))}\\
&+{\sum_{|I\cap\{n+2,n+3\}|=2}B_\LL(\J(L_I,\sigma_t))}\\
&+{\sum_{t\ge1,I\cap\{n+2,n+3\}=\emptyset }B_\LL(\J(L_I,\sigma_t))},
\end{align*}
and similarly
\begin{align*}
\sldim(\Cr(\LL))=\ldim(\Cr(\widetilde{\LL}))&
+{\sum_{|I\cap\{n+2,n+3\}|=1}B_{\Cr(\LL)}(\J(L_I,\sigma_t))}\\
&+{\sum_{|I\cap\{n+2,n+3\}|=2}B_{\Cr(\LL)}(\J(L_I,\sigma_t))}\\
&+{\sum_{t\ge1,I\cap\{n+2,n+3\}=\emptyset }B_{\Cr(\LL)}(\J(L_I,\sigma_t))}.
\end{align*}
Now it is not difficult to check that $\sldim(\Cr(\LL))=\sldim(\LL)$. 
Indeed
first of all we note that
\begin{equation} 
\label{formula k_C}
k_C^c=\sum_{i=1}^{n+3} m_i-c(n+1)-n(d-c)=k_C-c=m_{n+2}+m_{n+3}-d,
\end{equation}
where we denote by $k_C$ (resp.\ $k_C^c$) the multiplicity of containment of $C$ in the base locus of $\LL$ (resp.\ of $\Cr(\LL)$).

Let us also denote by $k_{I,\sigma_t}$ (resp.\ $k_{I,\sigma_t}^c$) the multiplicity of containment of 
$\J(L_I,\sigma_t)$ in the base locus of $\LL$ (resp.\ of $\Cr(\LL)$). We leave it to the reader to check by using \eqref{formula k_C} that the following holds.

If
$|I\cap\{n+2,n+3\}|=1$, then  $k_{I,\sigma_t}^c=k_{I,\sigma_t}$,
hence
$B_{\Cr(\LL)}(\J(L_I,\sigma_t))=B_{\LL}(\J(L_I,\sigma_t))$.

If $|I\cap\{n+2,n+3\}|=2$, 
then $k_{I,\sigma_t}^c=k_{I\setminus\{n+2,n+3\},t+1}$,
hence
$B_{\Cr(\LL)}(\J(L_I,\sigma_t))=B_{\LL}(\J(L_{I\setminus\{n+2,n+3\}},\sigma_{t+1}))$.

If $t\ge1$ and $I\cap\{n+2,n+3\}=\emptyset$, then 
$k_{I,\sigma_t}^c=k_{I\cup\{n+2,n+3\},t-1}$,
hence
$B_{\Cr(\LL)}(\J(L_I,\sigma_t))=B_{\LL}(\J(L_{I\cup\{n+2,n+3\}},\sigma_{t-1}))$.
\end{proof}

\subsection{Cases where Conjecture \ref{conjecture} holds}\label{evidences}

In this section we provide a list of evidences to Conjecture \ref{conjecture}. 

\subsubsection{Conjecture \ref{conjecture} holds for Cremona {transforms} of only linearly obstructed linear systems}

\begin{proposition}\label{Cremona linear}
Let $\ls$  be linear system with $n+3$ base points for which $k_C\le1$. Any linear system $\ls'$ that can be Cremona reduced to $\ls$
satisfies Conjecture \ref{conjecture}.
\end{proposition}
\begin{proof}
We have $\sldim(\ls')=\sldim(\ls)=\ldim(\ls)=\dim(\ls)$. The first equality follows from Lemma \ref{cremona-invariant}, the second follows from 
Definition \ref{new definition rnc} (see also Remark \ref{no multiple rnc}), the last inequality follows from Theorem \ref{theorem n+3}.
\end{proof}

\subsubsection{Conjecture  \ref{conjecture} is true for $n\le3$}\label{low dimension conj true}

\begin{proposition}
Conjecture \ref{conjecture} holds for $n=2$.
\end{proposition}
\begin{proof}
Set  $\LL=\LL_{2,d}(m_1,\ldots,m_5)$.
It is a well-known fact that the Segre-Harbourne-Gimigliano-Hirschowitz Conjecture holds for five points.
Moreover, from Riemann-Roch Theorem on the blow-up projective plane it follows that
$$\dim(\ls)=\binom{d+2}{2}-\sum_1^5\binom{m_i+1}2+\sum_{i,j}\binom{m_i+m_j-d}2+\binom{k_C}2$$
and one can easily verify that the right-hand side of the above is $\sldim(\ls)$.
\end{proof}

\begin{proposition}
Conjecture \ref{conjecture} holds true in $n=3$.
\end{proposition}
\begin{proof}
Let $\LL$ be a linear system in $\PP^3$. 
If $\LL$ is Cremona reduced, then it is linearly non-special by \cite[Theorem 5.3]{devolder-laface},
that is $\dim(\LL)=\ldim(\LL)$.
On the other hand since $k_C\le0$ by Remark \ref{no multiple rnc}
we have $\sldim(\LL)=\ldim(\LL)$ and this concludes the proof in this case.

Assume now that $\LL$ is not Cremona reduced and denote by $\Cr(\ls)$ the corresponding Cremona reduced linear system.
We have 
 $\sldim(\LL)=\sldim(\LL')
=\dim(\LL')=\dim(\LL)$. The first equality follows from Lemma \ref{cremona-invariant}, the second follows from the previous case
and the last one from \eqref{Cremona preserves dim}.
\end{proof}

\subsubsection{Families of homogeneous linear systems that satisfy Conjecture \ref{conjecture}}\label{multiples secants} 

Consider the following family of linear systems 
$$\ls(t,a):=\ls_{2t,a(t+1)}((at)^{2t+3}),$$
for all $t,a\ge1$. 
Notice that in the case $a=1$, the linear system has one section that is $\sigma_t$, see Section \ref{section divisorial cones}.

\begin{proposition}\label{multiples of secant}
For any $t, a\ge1$, $\ls=\ls(t,a)$ has one element, $a\sigma_t$. In particular it satisfies Conjecture \ref{conjecture}.    
\end{proposition}

\begin{proof}
The hypersurface $a\sigma_t$ belongs to $\ls(t,a)$ because it has degree $a(t+1)$ and multiplicity $at$ 
along the rational normal curve given by the $2t+3$ points, see discussion in Section \ref{secants}.

We prove by induction on $t$ that  $a\sigma_t$ is the unique element of $\ls(t,a)$ and that $\sldim(\LL)=1$.
If $t=1$, then the system $\ls=\LL_{2,2a}(a^5)$ has one section that consists of the multiple conic $a\sigma_1\subset\PP^2$.
Furthermore it is easy to compute that 
$\sldim(\LL)=1.$
 
Now assume that $t\ge 2$.
First,  by means of a Cremona transformation
we reduce to $\Cr(\ls(a,t))=\ls_{2t,at}((a(t-1))^{2t+1},(at)^2)$ and
{$\sldim(\ls(a,t))=\sldim(\Cr(\ls(a,t)))$},
by Lemma \ref{cremona-invariant}. 
Second,  we observe that $\textrm{Cone}(\Cr(\ls(a,t)))=\ls(t-1,a)$.
We conclude by induction and by Lemma \ref{cone-invariant}.
\end{proof}

\medskip
For every $b\ge1$, 
let us consider the following linear system
$$\ls(b):=\ls_{n,b(n+2)}((bn)^{n+3}).$$

\begin{proposition}\label{m=n}
Let $n\ge2$ and $b\ge1$. The linear system $\ls(b)$ has one element if $n$ is even and empty otherwise.
In particular $\ls(b)$ satisfies Conjecture \ref{conjecture}.
\end{proposition}

\begin{proof}
The proof is by induction on $n\ge1$.
If $n=1$, one has $\dim(\ls_{1,3b}(b^4))=0$.
If $n=2$, one has $\dim(\ls_{2,4b}((2b)^5))=1$. 

Now assume $n\ge3$. Notice that $\textrm{Cone}(\Cr(\ls(b)))=\ls_{n-2,bn}((b(n-2))^{n+1})$.
We conclude by induction on $n$, using Lemma \ref{cone-invariant} and Lemma \ref{cremona-invariant}.
\end{proof}

\medskip
Consider the family
$$\LL=\LL_{n,d}(n^{n+3}).$$

\begin{proposition}\label{m=n conjecture true}
The linear system $\ls$ satisfies Conjecture \ref{conjecture} for any $n\ge 2$ and $d\ge1$.
\end{proposition}
\begin{proof}
If $d\le n+1$, then  the system is empty by Theorem \ref{old effectivity}.
If $d=n+2$ we conclude by applying  Proposition \ref{m=n} in the case $b=1$.
If $d\ge n+3$, the statement follows from Theorem \ref{theorem n+3} and Remark \ref{no multiple rnc}.
\end{proof}

\end{document}